\documentclass[12 pt,twoside]{amsart}

\usepackage{amsmath,amsfonts,amsthm,amssymb}
\usepackage{multirow}
\usepackage{graphicx}

\newtheorem{theorem}{Theorem}[section]
\newtheorem{lemma}[theorem]{Lemma}
\newtheorem{proposition}[theorem]{Proposition}
\newtheorem{corollary}[theorem]{Corollary}

\newtheorem{observation}[theorem]{Observation}

\theoremstyle{definition}

\theoremstyle{remark}
\newtheorem{remark}[theorem]{Remark}

\numberwithin{equation}{section}

\begin{document}

\title[Generators of Lie groups and hypercyclicity]{Topological generators of abelian Lie groups and hypercyclic finitely generated abelian semigroups of matrices}

\author[H. Abels]{H. Abels}
\address{Fakult\"{a}t f\"{u}r Mathematik, Universit\"{a}t Bielefeld, Postfach 100131, D-33501 Bielefeld, Germany}
\email{abels@math.uni-bielefeld.de}

\author[A. Manoussos]{A. Manoussos}
\address{Fakult\"{a}t f\"{u}r Mathematik, SFB 701, Universit\"{a}t Bielefeld, Postfach 100131, D-33501 Bielefeld, Germany}
\email{amanouss@math.uni-bielefeld.de}
\urladdr{http://www.math.uni-bielefeld.de/~amanouss}
\thanks{During this research the second author was fully supported by SFB 701 ``Spektrale Strukturen und
Topologische Methoden in der Mathematik" at the University of Bielefeld, Germany. He is grateful for its generosity and hospitality.}

\date{}



\begin{abstract}
In this paper we bring together results about the density of subsemigroups of abelian Lie groups, the minimal number of topological generators of abelian Lie groups
and a result about actions of algebraic groups. We find the minimal number of generators of a finitely generated abelian semigroup or group of matrices with a dense
or a somewhere dense orbit by computing the minimal number of generators of a dense subsemigroup (or subgroup) of the connected component of the identity of its
Zariski closure.
\end{abstract}

\maketitle

\section{Introduction}

The concept of a hypercyclic abelian finitely generated semigroup of bounded linear operators was introduced by N.S. Feldman in  \cite{Feldman3}. It generalizes the
concept of a hypercyclic operator in linear dynamics. A bounded linear operator on a real or complex topological vector space is called \textit{hypercyclic} if the
semigroup generated by this operator has a dense orbit. Hypercyclicity of a single bounded linear operator is a phenomenon that occurs only in infinite dimensional
separable spaces, see e.g. \cite{BM} or \cite{GP}, but hypercyclicity of semigroups may occur also on finite dimensional vector spaces, see \cite{Feldman3},
\cite{{Ker}}. A problem raised by Feldman in \cite{Feldman3} is to find the minimal number of generators for an abelian semigroup of matrices which is hypercyclic.
More specific questions which arise naturally is to solve this problem for special types of matrices. In the same paper Feldman showed that this minimal number of
generators for a semigroup of diagonal  $n\times n$ matrices with complex entries is $n+1$. Recently there has been done much research about this subject. We mention
only \cite{Shkarin1}, \cite{A}, \cite{CoHaMa}, \cite{CoHaMa2}, \cite{AM1}, \cite{AM2}, \cite{CoPa} for the abelian case and \cite{AV}, \cite{J2} for the non-abelian
case. Our methods differ from those in the previously mentioned papers since we use methods from the theories of algebraic groups and algebraic actions.

In this paper we determine the minimal number of generators of a \textit{finitely generated group or semigroup of commuting matrices with real or complex entries}
with a dense or a somewhere dense orbit for several classes of matrices. We would like to point out that for this number \textit{there is no difference between dense
and somewhere dense orbits and also no difference between the group and the semigroup cases}. This follows from the following theorem which is the basic result of the
present paper.

Let $V$ be a finite dimensional real vector space and let $S$ be a subsemigroup of $GL(V)$. For a point $x\in V$ we say that $x$ has a \textit{somewhere dense orbit}
if the closure $\overline{S(x)}$ of its orbit $S(x)$ contains a non-empty open subset of $V$.

\begin{theorem} \label{sdenseLie1}
Let $S$ be a finitely generated commutative subsemigroup of $GL(V)$ and let $x\in V$ be a point which has a somewhere dense orbit. Let $G$ be the Zariski closure of
$S$ and let $G^0$ be its connected component of the identity with respect to the Euclidean topology. Then the orbit $G(x)$ of $x$ is an open subset of $V$, the
natural map $G\to G(x)$, $g \mapsto gx$, is a diffeomorphism and the closure of $S$ is a subgroup of $G$ and contains $G^0$.
\end{theorem}

We will obtain more precise information about the closure of the orbit $Sx$ in \cite{AbMa1}. In particular we will prove that $Gx$ is dense in $V$ for the complex
case, which implies the following.

\begin{corollary} \label{cor12}
If additionally $V$ has a structure as a complex vector space such that every $s\in S$ is complex linear, then $Sx$ is dense in $V$.
\end{corollary}

In this paper we show that the orbit $Sx$ is dense in $V$, if the number of generators of $S$ is minimal among all subsemigroups $S$ of $GL(V)$ having a somewhere
dense orbit, except if $\dim V$ is odd. For a precise formulation and further information about this see corollary \ref{cor3}.

After Theorem \ref{sdenseLie1} it remains to compute the minimal number of generators of a dense subsemigroup (or subgroup) of the abelian Lie group $G^0$. This
number is one plus the dimension of the following group, namely $G^0$ modulo its maximal compact subgroup; see section 4. The results of tables \ref{table1} and
\ref{table2} and similar ones follow.

\begin{table}[htdp]
\caption{Minimum number of generators of an abelian (semi)group of real matrices which has a dense or a somewhere dense orbit.} \label{table1}

\begin{center}
\scalebox{0.67}{
\begin{tabular}{|c|c|c|c|c|c|}  \hline
\hline \multicolumn{6}{|c|}{\textbf{Minimum number of generators of an abelian (semi)group of real matrices}} \\
\hline \hline \multicolumn{2}{|c|}{\textbf{\phantom{.xxxxxx}Commuting\phantom{xxxxx.}}} & \multicolumn{2}{|c|}{\textbf{\phantom{..}Diagonal}\phantom{.}} & \multicolumn{2}{c|}{\textbf{Triangular non-diagonalizable}} \\
\hline  \multicolumn{2}{|c|}{{\boldmath $(n+2)/2$}, \textbf{if} {\boldmath $n$} \textbf{is even}} & \multicolumn{2}{|c|}{\multirow{2}{*}{\boldmath $n+1$}} & \multicolumn{2}{|c|}{\multirow{2}{*}{\boldmath $n+1$}}\\
\multicolumn{2}{|c|}{{\boldmath $(n+3)/2$}, \textbf{if} {\boldmath $n$} \textbf{is odd}} & \multicolumn{2}{|c|}{} & \multicolumn{2}{|c|}{}\\
\hline \hline \multicolumn{6}{|c|}{\textbf{Triangular Toeplitz non-diagonalizable}}  \\
\hline \multicolumn{6}{|c|}{\boldmath $n+1$} \\
\hline
\end{tabular}
}
\end{center}

\end{table}

\begin{table}[htdp]
\caption{Minimum number of generators of an abelian (semi)group of complex matrices which has a dense or a somewhere dense orbit.} \label{table2}

\begin{center}
\scalebox{0.64}{
\begin{tabular}{|c|c|c|c|c|c|}  \hline
\hline \multicolumn{6}{|c|}{\textbf{Minimum number of generators of an abelian (semi)group of complex matrices}} \\
\hline \hline \multicolumn{2}{|c|}{\textbf{\phantom{xxx}Commuting\phantom{xxx}}} & \multicolumn{2}{|c|}{\textbf{\phantom{.xxx.}Diagonal\phantom{.xxx}}} & \multicolumn{2}{c|}{\textbf{Triangular non-diagonalizable}} \\
\hline  \multicolumn{2}{|c|}{\boldmath $n+1$} & \multicolumn{2}{|c|}{\boldmath $n+1$} & \multicolumn{2}{|c|}{\boldmath $n+2$} \\
\hline \hline \multicolumn{6}{|c|}{\textbf{Triangular Toeplitz non-diagonalizable}}  \\
\hline \multicolumn{6}{|c|}{\boldmath $n+2$} \\
\hline
\end{tabular}
}
\end{center}

\end{table}

We would like to stress that for complex matrices an orbit of a point is dense if it is somewhere dense but this is not the case for real matrices, see
\cite{Feldman3}. In the same work Feldman gave a criterion for when a somewhere dense orbit is dense  in terms of spectral properties of the elements of the semigroup
\cite[Theorem 5.5]{Feldman3}. In section 5 we give an explanation for this difference between the real and the complex case based on a structure theorem for algebraic
groups and in a forthcoming paper \cite{AbMa1} will give more details.

The contents of the paper are as follows. In section 2 we show that a somewhere dense subsemigroup of a connected abelian Lie group is actually dense. In section 3 we
give the proof of theorem \ref{sdenseLie1}. This proof follows from a basic theorem about actions of algebraic groups \cite[$\S$ 3.18]{BoTi}. In section 4 we compute
the minimal number of generators of a dense subgroup and a dense subsemigroup of a connected abelian Lie group. This is an easy application of Kronecker's theorem. In
section 5 we give applications for the case of hypercyclic finitely generated abelian (semi)groups of matrices with various types of generators.

So in this paper we bring together results about the density of subsemigroups of abelian Lie groups (section 2), the minimal number of topological generators of
abelian Lie groups (section 4) and a result about actions of algebraic groups (section 3), to solve the problem of determining the minimal number of commuting
matrices such that the semigroup they generate has a (somewhere) dense orbit. The results of sections 2 and 4 may be of independent interest.

\section{A somewhere dense finitely generated subsemigroup of an abelian connected Lie group is dense}
\begin{theorem} \label{sdenseLie}
Let $G$ be an abelian Lie group whose connected component $G^0$ is of finite index in $G$. Let $S$ be a finitely generated subsemigroup of $G$, which is somewhere
dense in $G$, i.e. its closure $\overline{S}$ contains a non-empty open subset of $G$. Then $\overline{S}$ is a subgroup of $G$ and contains $G^0$.
\end{theorem}

We first deal with some special cases.

\begin{lemma}\label{sdenserealposneg}
Let $S$ be a subsemigroup of the additive group $\mathbb{R}$, which contains both a negative and a positive real number. Then $S$ is either a discrete cyclic subgroup
of $\mathbb{R}$ or $S$ is dense in $\mathbb{R}$. In particular, the closure of $S$ is a group.
\end{lemma}
\begin{proof}
Let $m_{+}:=\inf \{a\, ;\, a\in S, a>0\}$ and $m_{-}:=\inf \{|a|\, ;\, a\in S, a<0\}$. We first claim that $m_{+}=m_{-}$. If $a>0$ and $-b<0$ are elements of $S$ then
$a-nb\in S$ for every $n\in\mathbb{N}$. Hence there is an element $a'\in S$ with $0<a'\leq b$. So $m_{+} \leq m_{-}$. Similarly $-b+na\in S$ for every
$n\in\mathbb{N}$. Hence there is an element $-b'\in S$ with $0<b'\leq a$. Thus $m_{-} \leq m_{+}$.

Now there are two cases. If $m_{+}=m_{-}=0$ then $S$ is dense in $\mathbb{R}$. If $m_{+}=m_{-}=a>0$ we claim that $S= a \cdot \mathbb{Z}$. Suppose $b\in S$. Let
$b=na+r$ with $n\in\mathbb{Z}$ and $0\leq r<a$. Our claim is that $r=0$. By definition of $m_{+}=m_{-}$ there is an element $c$ of $S$ arbitrarily close to $-na$,
hence $b+c\in S$ is arbitrarily close to $r$, so $r$ must be zero.
\end{proof}

The case of the additive group $\mathbb{R}$ of theorem \ref{sdenseLie} follows from the following lemma.

\begin{lemma} \label{sdensereals}
A finitely generated subsemigroup of the additive group $\mathbb{R}$ is either dense in $\mathbb{R}$ or has no accumulation point in $\mathbb{R}$.
\end{lemma}
\begin{proof}
Let $S$ be a finitely generated subsemigroup of the additive group $\mathbb{R}$. If $S$ contains both positive and negative elements, then $S$ is either a discrete
cyclic subgroup of $\mathbb{R}$ or $S$ is dense in $\mathbb{R}$, by lemma \ref{sdenserealposneg}. If $S$ contains only non-negative elements, then there is a minimal
element among the positive generators, say $s_0$. If $t$ is the number of generators of $S$ then there are at most $t\cdot n$ elements of $S$ of absolute value $\leq
n\cdot s_0$. Similarly, if $S$ contains only non-positive elements.
\end{proof}

\begin{lemma} \label{sdenseconern}
Let $S$ be a finitely generated subsemigroup of $\mathbb{R}^n$. If $\overline{S}$ has non-empty interior, then the cone $C(S)$ spanned by $S$, i.e. $C(S):=\{ r\cdot
s\, ;\, r\in\mathbb{R}, r\geq 0, s\in S \}$ is dense in $\mathbb{R}^n$.
\end{lemma}
\begin{proof}
We first claim that the closure $\overline{C(S)}$ of $C(S)$ is actually a \textit{convex cone}. Let $a$, $b$ be elements of $S$. Then all non-negative integral linear
combinations of $a$ and $b$ are in $S$ hence all non-negative rational linear combinations of $a$ and $b$ are in $C(S)$ and hence all non-negative real linear
combinations of $a$ and $b$ are in $\overline{C(S)}$. It follows that the convex cone spanned by two elements of $C(S)$ is contained in $\overline{C(S)}$, therefore
$\overline{C(S)}$ is convex.

We now show that $\overline{C(S)}$ is actually all of $\mathbb{R}^n$. If not there is a linear form $l$ on $\mathbb{R}^n$ such that $l(\overline{C(S)})\geq 0$, by the
separating hyperplane theorem. The subsemigroup $l(S)$ of $\mathbb{R}$ has only non-negative values and, being finitely generated, must be discrete by lemma
\ref{sdensereals}. But $\overline{S}$ has non-empty interior hence so does $\overline{l(S)}\supset l(\overline{S})$, since $l$ is an open map. This contradicts our
hypothesis.
\end{proof}

\noindent \textit{Proof of Theorem \ref{sdenseLie}.}  We first show the theorem for the case $G=\mathbb{R}^n$. Let $U$ be a non-empty open subset of $\overline{S}$.
It suffices to show that $\overline{S}$ intersects $-U$, say $-u\in -U\cap \overline{S}$, $u\in U$, because then the semigroup $\overline{S}$ contains the
neighborhood $-u + U$ of zero and hence $\overline{S}=\mathbb{R}^n$. We show in fact that $-U\cap C(S)$ is contained in $\overline{S}$. Note that the set $-U\cap
C(S)$ is non-empty by lemma \ref{sdenseconern}. So let $-u\in -U\cap C(S)$. We may assume that $u\neq 0$. Then the closed subsemigroup $\mathbb{R} u\cap \overline{S}$
of the line $\mathbb{R} u$ contains points on either side of zero, hence it is either all of $\mathbb{R} u$ or has no accumulation point, by lemma
\ref{sdenserealposneg}. The latter case is impossible since $u$ is an inner point of $\overline{S}$. Thus $\mathbb{R} u\subset \overline{S}$.

Now let $G$ be a connected abelian Lie group. So there is a covering homomorphism $\pi:\mathbb{R}^n\to G$ whose kernel is a finitely generated discrete subgroup of
$\mathbb{R}^n$. Hence if $S$ is a finitely generated subsemigroup of $G$ whose closure contains a non-empty open set, then so is $\pi^{-1}(S)$ in $\mathbb{R}^n$.
Hence $\pi^{-1}(S)$ is dense in $\mathbb{R}^n$, thus $S$ is dense in $G$.

The general case is now reduced to the case of a connected group by the following lemma and the observation that the image of $\overline{S}$ in $G/G^0$, under the
quotient map, is a subsemigroup of the finite group $G / G^0$ and is hence a subgroup. \qed

\begin{lemma}
Let $G$ be an abelian Lie group whose connected component $G^0$ is of finite index in $G$. If $S$ is a finitely generated subsemigroup of $G$ which is somewhere dense
in $G$ then $S\cap G^0$ contains a finitely generated subsemigroup of $G^0$ which is somewhere dense in $G^0$.
\end{lemma}
\begin{proof}
Let $m$ be the number of connected components of $G$. Then $x\mapsto x^m$ is a covering homeomorphism $G\to G^0$ with finite kernel which maps $S$ to the subsemigroup
$S^m=\{ s^m\, ;\, s\in S\}$ of $G^0$, which is somewhere dense in $G^0$ and generated by $\{ s_1^m,\ldots ,s_t^m \}$ if $S$ is generated by $\{ s_1,\ldots ,s_t\}$.
\end{proof}

We shall need the following fact later on.

\begin{lemma} \label{genablie}
Let $G$ be an abelian Lie group with finite component group $G/G^0$.
\begin{enumerate}
\item[(a)] If $G$ has a dense sub(semi)group with $t$ generators then so does $G^0$.

\item[(b)] If $G^0$ has a dense sub(semi)group with $t_0$ generators and $G/G^0$ can be generated (as a group or semigroup, there is no difference) by $t_1$ elements
then $G$ can be generated by $\max (t_0,t_1)$ elements.
\end{enumerate}
\end{lemma}
\begin{proof}
(a) is proved as the preceding lemma.

(b) Let $\overline{g}=gG^0$ be a connected component of $G$. Suppose it has order $m(\overline{g})$ regarded as an element of the group $G/G^0$. Then given an element
$h\in G^0$ there is an element $g\in\overline{g}$ such that $g^m=h$ with $m=m(\overline{g})$. This follows from the fact that the group homomorphism $x\mapsto x^m$ of
$G$ to itself is a covering with finite kernel and hence maps connected components of $G$ onto connected components of $G$ and that the induced homomorphism $G/G^0
\to G/G^0$ is $\overline{x}\mapsto \overline{x}^m$. Let us call an element $g$ as above the $m(\overline{g})$-th root of $h$ and denote it by
$\sqrt[m(\overline{g})]{h}$. Our claim now follows from the following fact. Let $\{ h_1,\ldots ,h_{t_0}\}$ be a set of generators of a dense sub(semi)group of $G^0$
and let $\{ \overline{g}_1,\ldots ,\overline{g}_{t_1}\}$ be a set of generators of $G/G^0$. Define $t_2=\min (t_0,t_1)$ and $A_{\min}=\{
\sqrt[m(\overline{g}_1)]{h_1},\ldots,\sqrt[m(\overline{g}_{t_2})]{h_{t_2}}\}$. Then $A_{\min}\cup \{ h_{t_2+1},\ldots,h_{t_0}\}$ generates a dense sub(semi)group of
$G$ if $t_0\geq t_1=t_2$, and $A_{\min}\cup \{  g_{t_2+1},\ldots,g_{t_1}\}$ generates a dense sub(semi)group of $G$ if $t_1\geq t_0=t_2$, where $g_i$ is an arbitrary
element of $\overline{g}_i$ for $i>t_2$.
\end{proof}

\section{Proof of Theorem 1.1.}
\begin{observation} \label{observ}
Let $G$ be an abelian subgroup of $GL(V)$ and let $x\in V$ be a point which has a somewhere dense orbit. Then the isotropy group $G_y$ is trivial for every point $y$
of the orbit $G(x)$ of $x$.
\end{observation}
\begin{proof}
The isotropy group of a point $y=gx$ of the orbit $G(x)$ is $G_y=gG_xg^{-1}$, which equals $G_x$ since $G$ is abelian. So $h\in G_x$ fixes every point of the orbit
$G(x)$. But the orbit $G(x)$ spans $V$, since it is somewhere dense. Hence $h=1$.
\end{proof}

All the claims of theorem \ref{sdenseLie1}, except the dimension of $Gx$, follow from the following proposition and theorem \ref{sdenseLie}.

\begin{proposition} \label{prop32}
Let $S$ be a commutative subsemigroup of $GL(V)$ and let $G$ be its Zariski closure. Suppose that the orbit $Sx$ of $x\in V$ is somewhere dense in $V$. Then $G\to
Gx$, $g\mapsto gx$, is a diffeomorphism of $G$ onto an open subset of $V$. In particular, $S$ is somewhere dense in $G$.
\end{proposition}

In a later paper \cite{AbMa1} we will see that $Gx$ is actually dense in $V$.

\bigskip

\noindent \textit{Proof of Theorem \ref{sdenseLie1}.} A basic theorem about algebraic actions of algebraic groups says that the orbit $G(x)$ is open in its closure
$\overline{G(x)}$ and that the natural map $G\to G(x)$, $g\mapsto gx$, is a submersion; see \cite[$\S$ 3.18]{BoTi}. Thus in our case it follows that $G(x)$ contains a
non-empty open subset of $V$. Then $G(x)$ is open in $V$, by $G$-invariance. Furthermore, the natural map $G\to G(x)$ is a diffeomorphism, since it is bijective, by
observation \ref{observ}, a submersion by the theorem cited above and its tangent map at the identity element has as kernel the tangent space to $G_x$ which is
trivial. It follows that the subsemigroup $S$ of $G$ is somewhere dense in $G$. \qed

\section{The number of topological generators of a connected abelian Lie group}
Let $G$ be a topological group. Let us define $d_{gr}G$ and $d_mG$ to be the minimal number of elements of a subset $A$ of $G$ such that the group, resp. semigroup,
generated by $A$ is dense in $G$. The subscript $m$ was chosen because a semigroup is sometimes called a monoid.

Let $G$ be a connected abelian Lie group. Then $G$ contains a unique maximal compact subgroup $T$. The subgroup $T$ is actually a torus, i.e. a compact connected
abelian Lie group. Let $d$ be the dimension of $G/T$. The number $d$ is called, occasionally, the \textit{non-compact dimension} of $G$.

\begin{theorem}  \label{conablie}
\[
d_{gr}G=d_mG=d+1
\]
unless $G$ is the trivial group.
\end{theorem}

The proof is broken up into a series of lemmata.

\begin{lemma}[Kronecker] \label{torus}
For a torus $G\neq \{ 1\}$ we have $d_{gr}G=d_mG=1$. More precisely, if $G=\mathbb{R}^n / \mathbb{Z}^n$ then for a vector $v=(r_1,\ldots,r_n)\in\mathbb{R}^n$ the
monoid generated by $v\,\, \text{mod}\, \mathbb{Z}^n$ is dense in $G$ if and only if the elements $1,r_1,\ldots,r_n$ in $\mathbb{R}$ are linearly independent over
$\mathbb{Q}$.
\end{lemma}
\begin{proof}
Let $H$ be the closure of the subsemigroup of $G$ generated by $v\,\, \text{mod}\, \mathbb{Z}^n$. Then $H$ is a subgroup of $G$ since every compact subsemigroup of a
topological group is actually a group. Then by Pontryagin duality $H$ is the intersection of the kernels of all continuous homomorphisms
$\overline{\varphi}:\mathbb{R}^n / \mathbb{Z}^n\to \mathbb{R} / \mathbb{Z}$ which vanish on $H$. Now $\overline{\varphi}$ is induced by a linear map
$\varphi:\mathbb{R}^n\to\mathbb{R}$ with the property that $\varphi(\mathbb{Z}^n)\subset\mathbb{Z}$, so $\varphi(x_1,\ldots,x_n)=\alpha_1x_1+\ldots +\alpha_nx_n$ with
$\alpha_i\in\mathbb{Z}$ for $(x_1,\ldots,x_n)\in\mathbb{R}^n$. But $\varphi=0$ is the only such map which maps $v=(r_1,\ldots,r_n)$ to an element of $\mathbb{Z}$ if
and only if $1,r_1,\ldots,r_n$ are linearly independent over $\mathbb{Q}$.
\end{proof}

\begin{lemma} \label{realgen}
$d_{gr}(\mathbb{R}^n)=d_m(\mathbb{R}^n)=n+1$ for $n>0$.
\end{lemma}
\begin{proof}
A subgroup of $\mathbb{R}^n$ generated by at most $n$ vectors is clearly not dense in $\mathbb{R}^n$. On the other hand, we will show that the subsemigroup of
$\mathbb{R}^n$ generated by a set $A$ of the following form is dense in $\mathbb{R}^n$. Let $A=\{ v_1,\ldots,v_{n+1}\}$, where $v_1,\ldots,v_n$ is a basis of the
vector space $\mathbb{R}^n$ over $\mathbb{R}$ and $v_{n+1}=r_1v_1+\ldots +r_nv_n$ and the elements $1,r_1,\ldots,r_n$ in $\mathbb{R}$ are linearly independent over
$\mathbb{Q}$ and all $r_i$ are negative. To prove this we may assume that  $v_1,\ldots,v_n$ is the standard basis of $\mathbb{R}^n$. By lemma \ref{torus} the
subsemigroup of $\mathbb{R}^n / \mathbb{Z}^n$ generated by $v_{n+1}\,\, \text{mod}\, \mathbb{Z}^n$ is dense in $\mathbb{R}^n / \mathbb{Z}^n$. So for each $n$-tuple
$x=(x_1,\ldots,x_n)\in\mathbb{R}^n$ and every $\varepsilon >0$ there is an integer $m\in\mathbb{N}$ and an element $\gamma\in\mathbb{Z}^n$ such that $\|
mv_{n+1}+\gamma -x\| <\varepsilon$ for a fixed norm $\| \cdot \|$ on $\mathbb{R}^n$. We may assume that $m$ is large by applying the preceding lemma to a multiple of
$v_{n+1}$. Then, if $x_i\geq 0$ for $i=1,\ldots,n$, it follows that the coordinates of $\gamma$ are also non-negative, if $\varepsilon$ is sufficiently small. So
$\gamma$ is an element of the subsemigroup $\mathbb{N}^n$ of $\mathbb{R}^n$ generated by $v_1,\ldots,v_n$. Thus every $x=(x_1,\ldots,x_n)\in\mathbb{R}^n$ with
$x_i\geq 0$ is contained in the closure $H$ of the subsemigroup of $\mathbb{R}^n$ generated by $A$. But if $x\in H$ then $x+mv_{n+1}\in H$ for every $m\in\mathbb{N}$.
Since all the coordinates of $v_{n+1}$ are negative, by hypothesis, every vector $y\in\mathbb{R}^n$ is of the form $y=mv_{n+1}+x$, for some $m\in\mathbb{N}$ and all
coordinates of $x$ non-negative. Hence $H=\mathbb{R}^n$.
\end{proof}

\noindent \textit{Proof of Theorem \ref{conablie}.}  Clearly $d_mG\geq d_{gr}G >d$, since $G/T$ is homeomorphic to $\mathbb{R}^d$. The universal covering group
$\widetilde{G}$ of $G$ is isomorphic to $\mathbb{R}^n$ for some $n$ and $\pi_1G$ is isomorphic to a lattice in a subspace of dimension $m=n-d$. Take a basis
$v_1,\ldots,v_n$ of $\mathbb{R}^n$ containing a basis of $\pi_1G$ and let  $A=\{ v_1,\ldots,v_n,v_{n+1}\}$ with $v_{n+1}=r_1v_1+\ldots +r_nv_n$, with $r_i<0$ and
$1,r_1,\ldots,r_n$ linearly independent over $\mathbb{Q}$. Then the subsemigroup of $\mathbb{R}^n$ generated by $A$ is dense in $\mathbb{R}^n$, hence the subsemigroup
of $G=\widetilde{G} / \pi_1G$ generated by the $n+1-m=d+1$ non-zero images of $A$ is dense in $G$ and this finishes the proof. \qed

\medskip

Here is a coordinate free formulation for the conditions on a set $A$ to generate a dense subsemigroup of $\mathbb{R}^n$ which follows from the proof above by
choosing appropriate coordinates.

\begin{corollary} \label{dsubV}
Let $V$ be a finite dimensional vector space over $\mathbb{R}$. A finite subset $A$ of $V$ generates a dense subsemigroup of $V$ if and only if the following two
conditions hold:
\begin{itemize}
\item[(a)] The convex hull of $A$ contains $O$ in its interior.

\item[(b)] The zero form is the only $\mathbb{R}$-linear form $l$ on $V$ with the property that $l(A)\subset\mathbb{Z}$.
\end{itemize}
\end{corollary}

\begin{remark}
In the previous corollary, condition (a) is equivalent to
\begin{itemize}
\item[(a$'$)] The zero form is the only $\mathbb{R}$-linear form $l$ on $V$ with the property that $l(A)\subset [0,+\infty)$.
\end{itemize}
\end{remark}

A more computational version of corollary \ref{dsubV} for the case that the cardinality of $A$ is $1+\dim V$ is the following:

\begin{remark} \label{subGrem}
A set $\{ v_1,\ldots,v_{n+1}\}$ of $n+1$ elements of an $n$-dimensional real vector space $V$ generates a dense subsemigroup of $V$ if and only if the following two
conditions hold:
\begin{itemize}
\item[(a)] The vectors $v_1,\ldots,v_n$ are linearly independent over $\mathbb{R}$.

\item[(b)] For the vector $v_{n+1}=\alpha_1v_1+\ldots +\alpha_nv_n$ we have $\alpha_i <0$ for all $i=1,\ldots,n$ and $1,\alpha_1,\ldots,\alpha_n$ are linearly independent
over $\mathbb{Q}$.
\end{itemize}
\end{remark}

We will also need the more general case of an arbitrary connected abelian Lie group.

Let $G$ be a connected abelian Lie group and let $\mathfrak{g}$ be its Lie algebra, a real vector space of dimension $n$, say. The exponential map
$\exp:\mathfrak{g}\to G$ is the universal covering homomorphism of $G$. Let $T$ be the maximal torus of $G$ and let $\mathfrak{t}=\exp^{-1}(T)$ be its Lie algebra.
Then $\Gamma:=\ker \exp$ is a lattice in $\mathfrak{t}$. Let $v_1,\ldots,v_t$ be a basis of this lattice.

\begin{corollary} \label{subGexp}
Let $\{ v_{t+1},\ldots,v_{n+1}\}$ be a set of elements of $\mathfrak{g}$. Then the subsemigroup of $G$ generated by $\exp(v_{t+1}),\ldots,\exp(v_{n+1})$ is dense in
$G$ if and only if the following two conditions hold:
\begin{itemize}
\item[(a)] The vectors $v_1,\ldots,v_n$ form a basis of $\mathfrak{g}$ over $\mathbb{R}$.

\item[(b)] The vector $v_{n+1}=\alpha_1v_1+\ldots +\alpha_nv_n$ has the following properties: $\alpha_i <0$ for $i=t+1,\ldots,n$ and $1,\alpha_1,\ldots,\alpha_n$ are
linearly independent over $\mathbb{Q}$.
\end{itemize}
\end{corollary}
\begin{proof}
Let $S$ be the subsemigroup of $\mathfrak{g}$ generated by $v_{t+1},\ldots,v_{n+1}$. Then $S+\Gamma$ is dense in $\mathfrak{g}$ if the conditions (a) and (b) are
satisfied, by remark \ref{subGrem}, hence $\exp(S)$ is dense in $G$. Conversely, if $S$ is dense in $G$, then its image in $G/T$ is dense, hence
$v_{t+1},\ldots,v_{n}$ are linearly independent modulo the span of $\Gamma$ and $\alpha_i <0$ for $i>t$. Assume there is a linear relation between
$1,\alpha_1,\ldots,\alpha_n$ over $\mathbb{Q}$. So there are numbers $b_1,\ldots,b_{n+1}$ in $\mathbb{Q}$, not all of them zero, such that $b_1\alpha_1 +\ldots
+b_n\alpha_n=b_{n+1}$. We may assume that $b_1,\ldots,b_{n+1}$ are in $\mathbb{Z}$, by multiplying with an appropriate non-zero integer. Now define a linear form
$l:\mathfrak{g}\to\mathbb{R}$ by putting $l(u_i)=b_i$ for $i=1,\ldots,n$. Then $l(u_{n+1})=b_{n+1}$. So $l(S+\Gamma)\subset\mathbb{Z}$ and hence $l$ induces a
non-zero homomorphism $\overline{l}:\mathfrak{g}/\Gamma \to \mathbb{R}/\mathbb{Z}$ whose kernel contains $S$. So $\exp{S}$ is not dense in $G$, which is isomorphic to
$\mathfrak{g}/\Gamma$ via $\exp$.
\end{proof}

\section{Applications to finitely generated hypercyclic abelian semigroups and groups of matrices}
We first collect the facts we have proven so far in order to apply them to the various cases.

Let $V$ be a real vector space of dimension $n$ and let $S$ be a commutative sub(semi)group of $GL(V)$. Let $G$ be the Zariski closure of $S$ which is a closed
abelian subgroup of $GL(V)$ with finite component group and let $T$ be its maximal compact connected subgroup.

\begin{theorem} \label{somewheregen}
If $S$ has a somewhere dense orbit then $S$ has at least $n+1-\dim T$ generators.
\end{theorem}
\begin{proof}
The group $G$ has dimension $n$ by theorem \ref{sdenseLie1}, so $d_{gr}G^0=d_mG^0=n+1 -\dim T$ by theorem \ref{conablie} and the corresponding numbers for $G$ are not
smaller by lemma \ref{genablie}.
\end{proof}

To obtain the exact minimal number of generators we will apply lemma \ref{genablie}, which in the terminology of section 4 reads as follows.

\begin{corollary} \label{cor1}
Let $G$ be an abelian Lie group with finite component group. Then
\[
d_{gr}G=d_mG=\max (d_mG^0,d_mG/G^0)=\max (d_{gr}G^0,d_{gr}G/G^0).
\]
\end{corollary}
\begin{proof}
Lemma \ref{genablie} and the formula $d_{gr}G^0=d_mG^0$ of theorem \ref{conablie}.
\end{proof}

We will apply this for closed abelian subgroups $G$ of $GL(V)$ with finite component group and for which there is a vector $x\in V$ such that the natural map $G\to
G(x)$, $g \mapsto gx$, is a diffeomorphism onto an open subset of $V$.

\begin{corollary} \label{cor2}
Let $G$ be a closed subgroup of $GL(V)$ with finite component group and let $T$ be a maximal torus of $G$.
\begin{enumerate}
\item[(a)] Let $S$ be a commutative sub(semi)group of $G$ with a somewhere dense orbit. Then $S$ has at least $1 + \dim V -\dim T$ generators.

\item[(b)] Suppose there is a closed connected abelian subgroup $H$ of $G$ which contains $T$ and has an open orbit. Then there is a dense sub(semi)group $S$ of $H$
with $1 + \dim V -\dim T$ (commuting) generators. Every dense subsemigroup $S$ of $G$ has a somewhere dense orbit.
\end{enumerate}
\end{corollary}

As an application we compute the numbers in tables \ref{table1} and \ref{table2} of the introduction. \textit{Note that the numbers for groups are the same as for
semigroups.} This follows from theorem \ref{conablie}. It is also turns out that \textit{in all the cases we consider the minimal numbers are the same for a somewhere
dense orbit as for a dense orbit.} This is due to corollary \ref{cor1}.

Define for any subgroup $G$ of $GL(V)$ the number $m(G)$ to be the minimal number of commuting elements of $G$ with the property that the subsemigroup $S$ of $G$ they
generate has a somewhere dense orbit. If $G$ has no finitely generated commutative subsemigroup with a somewhere dense orbit we set $m(G)=\infty$. We have the
following results which we list in the form of a table; see table \ref{table3}. For the definition of $T$ and $H$ in each case see corollary \ref{cor3}.

Details will be spelled out in the following corollaries.

For the subsemigroups generated by $m(G)$ elements we have very precise information in case of $G=GL(n,\mathbb{C})$ and $G=GL(n,\mathbb{R})$, see remark
\ref{subGrem}.

\begin{table}[htdp]
\caption{} \label{table3}

\begin{center}
\scalebox{0.9}{
\begin{tabular}{|c|c|c|c|}
\hline \hline \multicolumn{1}{|c|}{$G$} & \multicolumn{1}{|c|}{$m(G)$} & \multicolumn{1}{c|}{$T$}  & \multicolumn{1}{c|}{$H$}\\
\hline \hline \multicolumn{1}{|c|}{$GL(n,\mathbb{C})$} & \multicolumn{1}{|c|}{\boldmath $n+1$} & \multicolumn{1}{|c|}{$(S^1)^n$} & \multicolumn{1}{|c|}{$(\mathbb{C}^*)^n$} \\
\hline \multicolumn{1}{|c|}{$(\mathbb{C}^*)^n\subset G\subset GL(n,\mathbb{C})$} & \multicolumn{1}{|c|}{\boldmath $n+1$} & \multicolumn{1}{|c|}{$(S^1)^n$} & \multicolumn{1}{|c|}{$(\mathbb{C}^*)^n$} \\
\hline \multicolumn{1}{|c|}{$GL(2m,\mathbb{R})$} & \multicolumn{1}{|c|}{\boldmath $m+1$} & \multicolumn{1}{|c|}{$(SO(2))^m$} & \multicolumn{1}{|c|}{$(\mathbb{R}^*_{>0}\cdot SO(2))^m$} \\
\hline \multicolumn{1}{|c|}{$GL(2m+1,\mathbb{R})$} & \multicolumn{1}{|c|}{\boldmath $m+2$} & \multicolumn{1}{|c|}{$(SO(2))^m$} & \multicolumn{1}{|c|}{$(\mathbb{R}^*_{>0}\cdot SO(2))^m \times \mathbb{R}^*_{>0}$} \\
\hline \multicolumn{1}{|c|}{complex $n\times n$ Toeplitz} & \multicolumn{1}{|c|}{\boldmath $2n$} & \multicolumn{1}{|c|}{$S^1$} & \multicolumn{1}{|c|}{$G$} \\
\hline \multicolumn{1}{|c|}{real $n\times n$ Toeplitz} & \multicolumn{1}{|c|}{\boldmath $n+1$} & \multicolumn{1}{|c|}{$1$} & \multicolumn{1}{|c|}{$G^0$} \\
\hline \multicolumn{1}{|c|}{real triangular} & \multicolumn{1}{|c|}{\boldmath $n+1$} & \multicolumn{1}{|c|}{$1$} & \multicolumn{1}{|c|}{Toeplitz} \\
\hline
\end{tabular}
}
\end{center}
\end{table}

\begin{corollary} \label{cor3}
Let $G$ be $GL(n,\mathbb{C})$ or $GL(n,\mathbb{R})$ and let $H$ be the corresponding subgroup of table \ref{table3}. Let $S$ be a subsemigroup of $G$ generated by
$m(G)$ commuting elements and having a somewhere dense orbit. Then the following claims hold:
\begin{enumerate}
\item[(a)] $S$ is diagonalizable over $\mathbb{C}$.

\item[(b)] $S$ has an everywhere dense orbit unless $G=GL(2m+1,\mathbb{R})$ in which case for some point the orbit closure is either $\mathbb{R}^{2m+1}$ or a half
space of $\mathbb{R}^{2m+1}$.

\item[(c)] The closure of $S$ in $G$ is a conjugate of $H$, unless $G=GL(2m+1,\mathbb{R})$ in which case $\overline{S}$ contains a conjugate $H_1$ of $H$ and $H_1$ is
of index $1$ or $2$ in $\overline{S}$.
\end{enumerate}
We also have the information about the generators formulated in remark \ref{subGrem} and corollary \ref{subGexp}, in case $\overline{S}$ is connected.
\end{corollary}
\begin{proof}
Let us consider the following subgroups of $G$. For $G=GL(n,\mathbb{C})$ let $H:=T_{\mathbb{C}}$  be the group of all diagonal matrices in $GL(n,\mathbb{C})$ and let
$T$ be the subgroup of all those matrices in $H$ all of whose entries are of absolute value $1$. For $GL(2m,\mathbb{R})$ identify $\mathbb{R}^{2m}$ with
$\mathbb{C}^m$ and let $T$ and $H$ be the corresponding subgroups of $GL(m,\mathbb{C})\subset GL(2m,\mathbb{R})$. For $GL(2m+1,\mathbb{R})$ identify
$\mathbb{R}^{2m+1}$ with $\mathbb{C}^m \oplus \mathbb{R}$ and let $T$ be the corresponding subgroup of $GL(m,\mathbb{C}) \times GL(1,\mathbb{R})\subset
GL(2m+1,\mathbb{R})$. For $H$ take $T_{\mathbb{C}}\times \mathbb{R}^*_{>0}\subset GL(m,\mathbb{C}) \times GL(1,\mathbb{R})\subset GL(2m+1,\mathbb{R})$. Then $T$ and
$H$ have the properties of corollary \ref{cor2} and thus prove the first four rows of table \ref{table3}. In fact, every point of the space has a somewhere dense
orbit, except for the points of a finite union of vector subspaces of real codimension $1$ or $2$. The orbit is actually dense in $V$ except for
$G=GL(2m+1,\mathbb{R})$ where its closure is either the whole space or a half space.

Let now $S$ be a subsemigroup of $G$ generated by $m(G)$ commuting elements. Then the Zariski closure $G_1$ of $S$ has dimension $\dim_{\mathbb{R}}V$ by theorem
\ref{sdenseLie1} and contains a torus of dimension at least $\dim T$, by theorem \ref{conablie}, since $m(G)=1+\dim V -\dim T$. It follows that $G_1$ contains a
maximal torus of $G$. But all maximal tori of $G$ are conjugate. Thus, by conjugating, we may assume that $G_1$ contains $T$. Then $G_1$ is contained in the
centralizer of $T$. But the connected component of the centralizer of $T$ is $H$ in all cases we consider. More specifically, the centralizer of $T$ is $H$ and thus
connected, except for $G=GL(2m+1,\mathbb{R})$ where it is $T_{\mathbb{C}}\times \mathbb{R}^*$ and thus contains $H$ of index $2$. Now $\dim G_1=\dim H=\dim V$ and
$G_1^{0}\subset H$, so $H=G_1^{0}$. But the closure $\overline{S}$ of $S$ is a subgroup of $G_1$ and contains $G_1^{0}$, by theorem \ref{sdenseLie1}, so
$\overline{S}=H$ or, for $G=GL(2m+1,\mathbb{R})$, it also may happen that $\overline{S}$ contains $H$ of index $2$. This implies all the claims in corollary
\ref{cor3}.
\end{proof}

We now deal with the remaining cases of table \ref{table3}. Let $e_1,\ldots,e_n$ be the standard basis of $\mathbb{R}^n$ or $\mathbb{C}^n$. The \textit{backward
shift} $\sigma$ is the linear map with $\sigma (e_i)=e_{i-1}$ for $i=2,\ldots,n$ and $\sigma (e_1)=0$. A linear endomorphism is called a Toeplitz operator if it
commutes with $\sigma$. The following observation is easily checked.

\medskip

\noindent \textit{Observation.} For a linear endomorphism $A$ of $\mathbb{R}^n$ or $\mathbb{C}^n$ with representing matrix $a_{ij}$  the following properties are
equivalent:

\begin{enumerate}
\item[(a)] $A$ is Toeplitz.

\item[(b)] $A$ is a linear combination of $\sigma^i$, $i=0,\ldots,n-1$, where $\sigma^0$ is the identity.

\item[(c)] $a_{i+1,j+1}=a_{ij}$ for all $i,j$ and all $a_{ij}=0$ for $i>j$.
\end{enumerate}

\medskip

It follows that the group $G$ of all invertible Toeplitz operators is commutative and has a dense orbit. The group $G$ is connected in the complex case and has two
connected components in the real case, corresponding to positive resp. negative diagonal elements. Maximal tori in $G$ are $S^1$ in the complex case and the trivial
group in the real case. This proves the lines with the Toeplitz operators in table \ref{table3}.

For the group $G$ of triangular matrices, in the complex case $G$ contains the complex diagonal matrices, hence we are in line $2$, and in the real case there is no
torus. This shows that $m(G)\geq n+1$ in the real case. The opposite inequality $m(G)\leq n+1$ follows e.g. from the Toeplitz case.

In a forthcoming paper \cite{AbMa1} we will give more precise information on the closure of $Sx$, if $Sx$ is somewhere dense. It is a union of convex cones of a very
special type.

\bigskip

\noindent \textbf{Acknowledgements.} We would like to thank G.A. Margulis for pointing out proposition \ref{prop32}.

\end{document}